\documentclass[times,sort&compress,3p]{elsarticle}

\usepackage{amsmath,amsfonts,amssymb,amsthm,booktabs,color,epsfig,graphicx,hyperref,url}
\usepackage{tabularx}
\usepackage{amsthm}
\usepackage{here}
\usepackage{caption}
\usepackage{subcaption}
\usepackage{ascmac}
\usepackage{color}
\usepackage[labelfont=bf]{caption}
\usepackage{latexsym}

\newtheorem{theorem}{Theorem} 

\newtheorem{corollary}[theorem]{Corollary}
 
\newdefinition{rmk}{Remark} 
\newproof{pf}{Proof}
\newproof{pot}{Proof of Theorem \ref{thm2}}

\def\vec#1{\mbox{\boldmath{$#1$}}}

\begin{document}

\begin{frontmatter}



\title{Generalized heterogeneous hypergeometric functions and \\ the distribution of the largest eigenvalue of an elliptical Wishart matrix}


\author[a]{Aya Shinozaki}
\author[b]{Koki Shimizu}
\author[b]{Hiroki Hashiguchi}

\cortext[mycorrespondingauthor]{Corresponding author. Email address: \url{1420702@ed.tus.ac.jp}~(K. Shimizu).}

\affiliation[a]{organization={Chuo University},%Department and Organization
            addressline={1-13-27 Kasuga}, 
            city={Bunkyo-ku},
            postcode={112-8551}, 
            state={Tokyo},
            country={Japan}}
\affiliation[b]{organization={Tokyo University of Science},%Department and Organization
            addressline={1-3 Kagurazaka}, 
            city={Shinjuku-ku},
            postcode={162-8601}, 
            state={Tokyo},
            country={Japan}}
%\affiliation[c]{organization={Department of Applied Mathematics, Tokyo University of Science},%Department and Organization
%            addressline={1-3 Kagurazaka}, 
%            city={Shinjuku-ku},
%            postcode={162-8601}, 
%            state={Tokyo},
%            country={Japan}}

\begin{abstract}
In this study, we derive the exact distributions of eigenvalues of a singular Wishart matrix under an elliptical model. 
We define generalized heterogeneous hypergeometric functions with two matrix arguments and provide convergence conditions for these functions.
The joint density of eigenvalues and the distribution function of the largest eigenvalue for a singular elliptical Wishart matrix are represented by these functions. Numerical computations for the distribution of the largest eigenvalue were conducted under the matrix-variate $t$ and Kotz-type models. 
\end{abstract}

\end{frontmatter}


%\section{}
%\label{}
\section{Introduction}
The distribution theory of eigenvalues of a Wishart matrix has been studied under the assumption of normality.
Under this assumption, the hypergeometric functions of matrix arguments introduced by Constantine~\cite{C1963} were used to express many distributions of eigenvalues of central or noncentral Wishart matrices. 
%The hypergeometric functions of matrix arguments are an infinite series of zonal polynomials introduced by James~\cite{J1960}.
%non-singular case
The exact distributions of the largest and smallest eigenvalues of a Wishart matrix were derived by Sugiyama~\cite{S1967} and Khatri~\cite{K1972}, respectively. 
An elliptically contoured distribution, which is a more general assumption than normality, has also been well studied (Fang and Zhang~\cite{FZ1990} and Fang et al.~\cite{F1990}, among others.
Matrix-variate elliptically contoured distributions include a matrix-variate normal, Pearson type VII, Kotz type, Bessel, and Jensen-logistic distributions. 
The generalized hypergeometric functions that are useful for the derivation of distributions of eigenvalues under the elliptical model was defined by D{\'i}az-Garc{\'i}a and Caro-Lopera~\cite{GarciaCaro2008}.
Caro-Lopera et al.~\cite{C2014} derived the density of an elliptical Wishart matrix and provided the exact distribution for testing the equality of covariance matrices.
Furthermore, Caro-Lopera et al.~\cite{C2016} provided the exact distributions of the extreme eigenvalues of an elliptical Wishart matrix.
These results of eigenvalue distributions cover the classical results under the Gaussian model as a special case. 
Shinozaki et al.~\cite{S2018} provided the alternative approach for the derivation of the exact distribution of the largest eigenvalue and conducted numerical experiments under the matrix variable $t$ model.
The largest and smallest eigenvalues of a ratio of two elliptical Wishart matrices were also given by Shinozaki and Hashiguchi~\cite{SH2018} 
%singular case

In the case of a singular Wishart matrix, Uhlig~\cite{U1994} provided useful Jacobians for the transformation of singular matrices and its density.
Its joint density of eigenvalues was given by Srivastava~\cite{S2003} with integrals over the Steifel manifold.
In the shape theory, D\'{i}az-Garc\'{i}a and Caro-Lopera\cite{GarciaCaro2013} provided the shape density relating the eigenvalues distribution of a singular Wishart matrix. 
Shimizu and Hashiguchi~\cite{S2021a} defined heterogeneous hypergeometric functions with two matrix arguments that are useful for deriving the distributions of eigenvalues of a singular random matrix.
The exact distributions of the largest eigenvalue of a singular Wishart and $F$ matrices were given by Shimizu and Hashiguchi~\cite{S2021a,S2021b}.

In this study, we show that the exact distributions of eigenvalues of a singular elliptical Wishart matrix are expressed in terms of generalized heterogeneous hypergeometric functions. 
In Section \ref{hhgf}, we introduce the matrix-variate elliptically contoured distribution and define the generalized heterogeneous hypergeometric functions. 
Furthermore, we provide the convergence condition for the generalized heterogeneous hypergeometric functions.
The exact distribution of the largest eigenvalue of a singular elliptical Wishart matrix is presented in Section \ref{exact}.
Our derivation is based on the method of Sugiyama~\cite{S1967}. 
In Section \ref{numerical}, we compute the distribution of the largest eigenvalue under the matrix-variate $t$ and Kotz-type models.  
%Finally, we consider some results related to the non-singular Wishart case in Section \ref{non-s}.

\section{Generalized heterogeneous hypergeometric function ${}_rP_s^{(m, n)}$}\label{hhgf}

 An $m\times n$ random matrix $X$ is said to have a matrix-variate elliptically contoured distribution 
 $\mathcal{E}_{m\times n}(M, \Sigma\otimes\Omega; h)$, if its density function is given as 
\begin{align}
g_{\tiny X}(X)=
\frac{1}{|\Sigma|^{n/2}|\Omega|^{m/2}}
h({\rm tr} \Sigma^{-1}(X-M)\Omega^{-1}(X-M)^\top), \label{eq:ellip}
\end{align}
where $M$ is the $m\times n$ mean matrix, $\Sigma$ is $m\times m$, $\Omega$ is $n\times n$, $\Sigma> 0$, and $\Omega> 0$,   
and the generator function $h$: $\mathbb{R}\to [0, \infty)$, satisfies $h(u)\in C^\infty$ with uniform convergence in $\mathbb{R}$.
If $X \sim \mathcal{E}_{m\times n}(\mathbf{0}, \Sigma\otimes I_n, h)$, 
where $M=\mathbf{0}$ and $\Omega = I_n$ in \eqref{eq:ellip},
then we call $W=XX^\top$ the elliptical Wishart matrix and write it as $W \sim \mathcal{EW}_m(n, \Sigma; h)$.
If $n<m$, then the Wishart matrix $W$ is called singular; otherwise, it is non-singular.
The singular elliptical Wishart matrix $W$ has $n$ positive eigenvalues and $ m-n $ zero eigenvalues. 
Using these positive eigenvalues, say, $\ell_1, \dots, \ell_n$, 
it has the spectral decomposition as $W = H_1 L_1 H_1^\top$, 
where $L_1 = \mathrm{diag}(\ell_1, \dots, \ell_n)$, $\ell_1>\cdots>\ell_n>0$
 and the $m \times n$ matrix $H_1$ is satisfied by $H_1^\top H_1=I_n$.
The set of all such $m\times n$ matrices $H_1$ with orthonormal columns is called the Stiefel manifold $V_{n, m}$, defined by
\begin{align*}
V_{n, m}=\{H_1 \in \mathbb{R}^{m \times n} \mid H_1^\top  H_1=I_n \},
\end{align*}
where $n \le m$. 
D{\' i}az-Garc{\' i}a and Guti\'errez-J\'aimez~\cite{G2006} gave the density function of a singular elliptical Wishart matrix $W$  as 
\begin{align}
\frac{\pi^{n^2/2}}{|\Sigma|^{n/2}\Gamma_n(n/2)}|L_1|^{(n-m-1)/2}h(\mathrm{tr}\Sigma^{-1}W),
\label{eq:densityW}
\end{align}
where the multivariate gamma function is
\begin{align*}
\Gamma_m(c)=\pi^{m(m-1)/4}\prod_{i=1}^{m}\Gamma \biggl(c-\frac{i-1}{2}\biggl), \ \ \mathrm{Re}(c)>m-1. %ガンマ関数の定義
\end{align*}
Because $h(x) \in \mathbb{C}^\infty$, the Maclaurin expansion of $h$ is expressed as
\begin{align}
 h(x) = \sum_{k=0}^\infty \dfrac{h^{(k)}(0)}{k !} x^k. \label{eqn:Mc-h}
\end{align} 
Furthermore, for an $m \times m$ symmetric matrix $X$, the function $h(\mathrm{tr} X)$ can be expanded  
by zonal polynomials $\mathcal{C}_\kappa(X)$ associated with a partition $\kappa$ of $k$.
For a positive integer $k$, let $\kappa=(\kappa_1, \dots, \kappa_m)$ denote a partition of $k$ with 
$\kappa_1\ge\cdots\ge\kappa_m\ge 0$ and $\kappa_1+\cdots+\kappa_m=k$. 
The set of all partitions with lengths not longer than $m$ is denoted by 
$P_m^k=\{\kappa=(\kappa_1, \dots, \kappa_m) \mid 
\kappa_1+\cdots+\kappa_m=k, \kappa_1\ge\kappa_2\ge\cdots\ge\kappa_m\ge 0\}$. 
The Pochammer symbol for a partition $\kappa$ is defined as 
$(\alpha)_\kappa=\prod_{i=1}^n\{\alpha-(i-1)/2\}_{\kappa_i}$, 
where $(\alpha)_k=\alpha(\alpha+1)\cdots(\alpha+k-1)$ and $(\alpha)_0=1$. 
For the $m\times m$ symmetric matrix $X$ with eigenvalues $x_1, \dots, x_m$,
the zonal polynomial $\mathcal{C}_\kappa(X)$ is defined as a symmetric polynomial in $x_1, \dots, x_m$. 
See p.227 of Muirhead~\cite{R1982} for details.
Shimizu and Hashiguchi~\cite{S2021a} showed
\begin{align} 
\int_{V_{n, m}} \mathcal{C}_\kappa(X H_1 Y H_1^\top) (d H_1) = 
\frac{\mathcal{C}_\kappa(X) \mathcal{C}_\kappa(Y)}{\mathcal{C}_\kappa(I_m)}
\label{eqn:int-C_HXH}
\end{align}
for an $m \times m$ symmetric matrix $X$, and an $n \times n$ symmetric matrix $Y$, 
where $(d H_1)$ is the differential form of $V_{n, m}$, such that
$$
\int_{V_{n, m}} (d H_1) = 1, \quad 
(d H_1) = \dfrac{1}{\mathrm{Vol}(V_{n, m}) }(H_1^\top dH_1), \quad (H_1^\top dH_1)=\bigwedge_{j=i+1}^m \bigwedge_{i=1}^n \mathbf{h}_j^\top d\mathbf{h}_i, 
$$ 
\begin{align}
\mathrm{Vol}(V_{n, m}) =\int_{V_{n,m} }(H_1^\top dH_1)= \dfrac{2^n\pi^{mn/2}}{\Gamma_n(m/2)}\label{vol_Steifel}
\end{align}
and $(H_1 \mid H_2) = (\mathbf{h}_1,  \dots, \mathbf{h}_n \mid  \mathbf{h}_{n+1}, \dots, \mathbf{h}_{m}) \in O(m)$.

From \eqref{eqn:Mc-h} and the property of zonal polynomials, we can define $_{0}P_{0}(h^{(k)}(0): X)$ as
\begin{align}
_{0}P_{0}(h^{(k)}(0): X) =  \sum_{k=0}^\infty \dfrac{h^{(k)}(0)}{k !} (\mathrm{tr} X)^k =  \sum_{k=0}^\infty \dfrac{h^{(k)}(0)}{k !} 
\sum_{\kappa \in P_m^k} \mathcal{C}_\kappa(X),
\label{eqn:0P0-series}
\end{align}
which is an infinite series expression of $h(\mathrm{tr} X)$.
If $h(x)=\exp(x)$, then we have $_{0}P_{0}(1: X) = \exp(\mathrm{tr} X) = {}_{0} F_{0}(X)$, where $ {}_{0} F_{0}(X)$ is the 
hypergeometric function with a matrix argument of type $(0,0)$.
We also define
\begin{align}
 {}_{0} P_{0}^{(m,n)}(h^{(k)}(0): X, Y) = \int_{V_{n, m}} {}_{0}P_{0}(h^{(k)}(0): X H_1 Y H_1^\top) (d H_1). 
 \label{eqn:def-0P0} 
\end{align}
for an $m \times m$ symmetric matrix $X$ and an $n \times n$ symmetric matrix $Y$.
Then, the function ${}_{0} P_{0}^{(m,n)}(h^{k}(0): X, Y)$ can be expanded in terms of zonal polynomials according to the following theorem:

\begin{theorem} \label{thm:0P0-C} Let $h(x) \in \mathbb{C}^\infty$ with uniform convergence in $\mathbb{R}$. 
For an $m \times m$ symmetric matrix $X$ and an $n \times n$ symmetric matrix $Y$, where $m \ge n$, 
the function $_{0}P_{0}^{(m, n)}(h^{(k)}(0): X, Y)$ defined in \eqref{eqn:def-0P0} is an infinite series of zonal polynomials as  
\begin{align*}
 {}_{0} P_{0}^{(m,n)}(h^{(k)}(0): X, Y) = \sum_{k=0} \dfrac{h^{(k)}(0)}{k !} \sum_{\kappa \in P_n^k}\dfrac{ \mathcal{C}_\kappa(X) \; \mathcal{C}_\kappa(Y)}
{\mathcal{C}_\kappa(I_m)}.
\end{align*}
\end{theorem}
\begin{proof}
From the uniform convergence of $h$ and \eqref{eqn:0P0-series}, 
the right-hand side of \eqref{eqn:def-0P0} can be integrated term by term as
\begin{align*}
 {}_{0} P_{0}^{(m,n)}(h^{(k)}(0): X, Y) 
 &= \int_{V_{n, m}} {}_{0}P_{0}(h^{(k)}(0): X H_1 Y H_1^\top) (d H_1)\\
 &=\sum_{k=0}^\infty\frac{h^{(k)}(0)}{k!}\sum_{\kappa\in P_n^k}\int_{V_{n, m}}\mathcal{C}_\kappa(XH_1YH_1^\top)(dH_1)\\
 &=\sum_{k=0}^\infty\frac{h^{(k)}(0)}{k!}\sum_{\kappa\in P_n^k}
 \frac{\mathcal{C}_\kappa(X)\mathcal{C}_\kappa(Y)}{\mathcal{C}_\kappa(I_m)}.
\end{align*}
The third term above is obtained by \eqref{eqn:int-C_HXH}.
\end{proof}

For an $m \times m$ positive definite $X$, 
we define ${}_{1}P_{1}(h^{(k)}(0): a; c; X )$ as the integral of the multivariate beta distribution as follows: 
\begin{align}
 {}_{1}P_{1}(h^{(k)}(0): a; c; X ) 
 &= \dfrac{\Gamma_m(c)}{\Gamma_m(a)\; \Gamma_m(c-a)}  \int_{\mathbf{0} < Y < I_m}
 {}_{0}P_{0}(h^{(k)}(0): XY ) |Y |^{a - \frac{m+1}{2}}  |I_m - Y |^{c-a - \frac{m+1}{2}} (dY), \label{eqn:def-1P1}
\end{align}
where $\mathrm{Re}(a) > \frac{1}{2}(m-1)$,  $\mathrm{Re}(c) > \frac{1}{2}(m-1)$,  and  $\mathrm{Re}(c-a) > \frac{1}{2}(m-1)$. 
%\textcolor{red}{This detail is given by D{\' i}az-Garc{\' i}a and Caro-Lopera~(2008) \cite{GarciaCaro2008}. }
Then, the function ${}_{1}P_{1}(h^{(k)}(0): a; c; X )$ in \eqref{eqn:def-1P1} can be expressed as
\begin{align}
\label{eqn:1P1-series}
{}_{1}P_{1}(h^{(k)}(0): a; c; X ) = \sum_{k=0}^\infty \dfrac{h^{(k)}(0)}{k !} \sum_{\kappa \in P_m^k} \dfrac{(a)_\kappa}{(c)_\kappa} \mathcal{C}_\kappa(X).
\end{align}
The above function \eqref{eqn:1P1-series} was firstly defined by D{\'i}az-Garc{\'i}a and Caro-Lopera~\cite{GarciaCaro2008}. 
If $h(x)=\exp(x)$, then we have ${}_{1}P_{1}(1; a; c; X ) = {}_1F_1(a; c; X)$ 
that is the confluent hypergeometric function of a matrix argument $X$.
Analogous to \eqref{eqn:def-0P0}, the generalized heterogeneous hypergeometric function of type $(1,1)$, 
${}_{1}P_{1}(h^{(k)}(0): a; c; X )$ is defined as
\begin{align}
 {}_{1} P_{1}^{(m,n)}(h^{(k)}(0): a; c; X, Y) = \int_{V_{n, m}} {}_{1}P_{1}(h^{(k)}(0): a; c; X H_1 Y H_1^\top) (d H_1) 
 \label{eqn:def-1P1-hetero}
\end{align}
for an $m \times m$ positive definite $X$ and an $n \times n$ positive definite $Y$, where $m \ge n$.
The following theorem holds in the same way as Theorem~\ref{thm:0P0-C}. 
\begin{theorem}  \label{thm:1P1-C}
Let $h(x) \in \mathbb{C}^\infty$ with uniform convergence in $\mathbb{R}$. 
For an $m \times m$ positive definite $X$ and an $n \times n$ positive definite $Y$, where $m \ge n$, 
the function $_{0}P_{0}^{(m, n)}(h^{(k)}(0): X, Y)$ defined in \eqref{eqn:def-1P1-hetero} is an infinite series of zonal polynomials as  
\begin{align*}
 {}_{1} P_{1}^{(m,n)}(h^{(k)}(0): a; c; X, Y) = \sum_{k=0} \dfrac{h^{(k)}(0)}{k !} \sum_{\kappa \in P_n^k}
 \dfrac{(a)_\kappa}{(c)_\kappa}
 \dfrac{ \mathcal{C}_\kappa(X) \; \mathcal{C}_\kappa(Y)}
{\mathcal{C}_\kappa(I_m)}.
\end{align*}
\end{theorem}
\begin{proof}
From the uniform convergence of $h$ and \eqref{eqn:def-1P1}, 
the right-hand side of \eqref{eqn:def-1P1-hetero} can be integrated term by term as
\begin{align*}
 {}_{1} P_{1}^{(m,n)}(h^{(k)}(0): a; c; X, Y)
 &= \int_{V_{n, m}} {}_{1}P_{1}(h^{(k)}(0): a; c; X H_1 Y H_1^\top) (d H_1)\\
 &=\sum_{k=0}^\infty\frac{h^{(k)}(0)}{k!}\sum_{\kappa\in P_n^k}\frac{(a)_\kappa}{(c)_\kappa}\int_{V_{n, m}}\mathcal{C}_\kappa(XH_1YH_1^\top)(dH_1)\\
 &=\sum_{k=0}^\infty\frac{h^{(k)}(0)}{k!}\sum_{\kappa\in P_n^k}\frac{(a)_\kappa}{(c)_\kappa}
 \frac{\mathcal{C}_\kappa(X)\mathcal{C}_\kappa(Y)}{\mathcal{C}_\kappa(I_m)}
\end{align*}
in the same proof of Theorem~\ref{thm:0P0-C}.
\end{proof}

Generally, if there exists 
\begin{align*}
{}_rP_s (h^{(k)}(0): \vec{\alpha}; \vec{\beta}; X)=
\sum_{k=0}^\infty \frac{h^{(k)}(0)}{k!}\sum_{\kappa\in P^k_m}\frac{(\alpha_1)_\kappa\cdots(\alpha_r)_\kappa}{
(\beta_1)_\kappa\cdots(\beta_s)_\kappa}
\mathcal{C}_\kappa(X)
\end{align*}
for an $m \times m$ symmetric matrix $X$, 
$\vec{\alpha}=(\alpha_1, \dots, \alpha_r)$ and $\vec{\beta}=(\beta_1, \dots, \beta_s)$, 
then the function ${}_rP_s^{(m, n)}(h^{(k)}(0): \vec{\alpha}; \vec{\beta}; X, Y)$ is defined by
 \begin{align*}
{}_rP_s^{(m, n)}(h^{(k)}(0): \vec{\alpha}; \vec{\beta}; X, Y)=
\int_{H_1\in V_{n, m}}{}_rP_s(h^{(k)}(0): \vec{\alpha}; \vec{\beta}; XH_1YH_1^\top )(dH_1)
\end{align*}
in addition to an $n \times n$ symmetric matrix $Y$.
From the uniform convergence of $h$ and \eqref{eqn:0P0-series}, 
function ${}_rP_s^{(m, n)}(h^{(k)}(0): \vec{\alpha}, \vec{\beta}; X, Y)$ has the following infinite series expansion:
\begin{align}
{}_rP_s^{(m, n)}(h^{(k)}(0): \vec{\alpha}; \vec{\beta}, X; Y)=
\sum_{k=0}^\infty \frac{h^{(k)}(0)}{k!}\sum_{\kappa\in P^k_u}\frac{(\alpha_1)_\kappa\cdots(\alpha_r)_\kappa}{(\beta_1)_\kappa\cdots(\beta_s)_\kappa}
\frac{\mathcal{C}_\kappa(X)\mathcal{C}_\kappa(Y)}{\mathcal{C}_\kappa(I_u)}.
\label{eqn:sPs-series}
\end{align}
It is clear that 
$
{}_rP_s(h^{(k)}(0): \vec{\alpha}; \vec{\beta}; X)=
{}_rP_s^{(m, m)}(h^{(k)}(0): \vec{\alpha};\vec{\beta}; X, I_m). \label{eq:mm}
$
To discuss the convergence condition of \eqref{eqn:sPs-series},
we provide the following theorem.
\begin{theorem} \label{thm:conv-rPs}
Suppose that $h(x) \in C^\infty$ with uniform convergence in $\mathbb{R}$, and there exists a constant $M < \infty$ such that  
$$
 M \ge \sup\{ 
  | h^{(k)}(0) | \mid k=1, 2, \dots
 \},
$$
then we have 
$$
  {}_rP_s(h^{(k)}(0): \vec{\alpha}; \vec{\beta}; X) \le M \; {}_r F_s(\vec{\alpha}; \vec{\beta}; X)
$$
for an $m \times m$ positive definite matrix $X$, $\vec{\alpha}=(\alpha_1, \dots, \alpha_r)$, and $\vec{\beta}=(\beta_1, \dots, \beta_s)$,
where ${}_r F_s(\vec{\alpha}; \vec{\beta}; X)$ is the hypergeometric function of a matrix argument $X$.
\end{theorem}
\begin{proof} It is clear that 
\begin{align*}
{}_rP_s (h^{(k)}(0): \vec{\alpha}; \vec{\beta}; X) &=
\sum_{k=0}^\infty \frac{h^{(k)}(0)}{k!}\sum_{\kappa\in P^k_m}\frac{(\alpha_1)_\kappa\cdots(\alpha_r)_\kappa}{
(\beta_1)_\kappa\cdots(\beta_s)_\kappa} \mathcal{C}_\kappa(X)
\\
& \le \sum_{k=0}^\infty \frac{M}{k!}\sum_{\kappa\in P^k_m}\frac{(\alpha_1)_\kappa\cdots(\alpha_r)_\kappa}{
(\beta_1)_\kappa\cdots(\beta_s)_\kappa} \mathcal{C}_\kappa(X)
= M \; {}_r F_s(\vec{\alpha}; \vec{\beta}; X).
\end{align*}
\end{proof}

Theorem~\ref{thm:conv-rPs} implies that the convergence condition of ${}_rP_s$
is almost the same as that of ${}_r F_s$.
If $h(y) =\exp(-y/2)/(2\pi)^{m n / 2}$, the function 
${}_rP_s^{(m, n)}(h^{(k)}(0): \vec{\alpha}; \vec{\beta}; X, Y)$ corresponds to the 
heterogeneous hypergeometric function of two matrix arguments ${}_r F_{s}^{(m, n)}(\vec{\alpha};\vec{\beta}; X, Y)$ 
introduced in Shimizu and Hashiguchi~\cite{S2021a}.

\section{Exact distribution of eigenvalues of a singular elliptical Wishart matrix}\label{exact}
In this section, we derive the joint density of the eigenvalues and the largest eigenvalue of a singular elliptical Wishart matrix. 
These results are an extension of the results from Shimizu and Hashiguchi~\cite{S2021a}.
\begin{theorem} \label{thm:joint-dis}
Let $W\sim\mathcal{EW}_m(n, \Sigma, h)$, where $n<m$. 
Then the joint density function of $\ell_1, \dots, \ell_n$ is given as 
\begin{align}
f(\ell_1, \dots, \ell_n) 
&=\frac{\pi^{n(n+m)/2}}{|\Sigma|^{n/2}\Gamma_n(n/2)\Gamma_n(m/2)}|L_1|^{(m-n-1)/2}
\prod_{i<j}^{n}(\ell_i-\ell_j)
\,{}_0P_0^{(m, n)}(h^{(k)}(0): \Sigma^{-1}, L_1), \label{eq:join}
\end{align}
where $L_1=\mathrm{diag}(\ell_1, \dots, \ell_n)$. 
\end{theorem}
%%%%%%%%%%
\begin{proof}
The Jacobian of the spectral decomposition $W=H_1L_1H_1^\top$ was given by Uhlig~\cite{U1994} as
\begin{align*} 
(dW) =2^{-n} |L_1|^{m-n}\prod_{i <  j}^n(\ell_i-\ell_j)(H_1^\top dH_1)(dL_1). 
%\label{eqn:Jacov-W}
\end{align*} 
Using the above relationship, the joint density of $L_1$ and $H_1$ is obtained from \eqref{eq:densityW} as 
\begin{align} 
\frac{\pi^{{n+m}/2}}{|\Sigma|^{n/2}\Gamma_n(n/2)\Gamma_n(m/2)}|L_1|^{(m-n-1)/2} \prod_{ i<j}^{n}(\ell_i-\ell_j) 
h(\mathrm{tr}\Sigma^{-1}H_1 L_1 H_1^\top).
 \label{eq:densityW-2}
\end{align} 
Furthermore, the Maclaurin expansion of $h(\cdot)$ in \eqref{eq:densityW-2} can be written as 
\begin{align*}
h(\mathrm{tr}\Sigma^{-1}H_1 L_1 H_1^\top)&=\sum_{k=0}^{\infty}\frac{h^{k}(0)}{k!}(\mathrm{tr}\Sigma^{-1}H_1 L_1 H_1^\top)^k\\
&=\sum_{k=0}^{\infty}\frac{h^{k}(0)}{k!}\sum_{\kappa\in P^k_m}\mathcal{C}_\kappa (\Sigma^{-1}H_1 L_1 H_1^\top).
\end{align*}
Hence, we get the joint density of $\ell_1, \dots, \ell_n$ as
\begin{align*}
\frac{\pi^{n(n+m)/2}}{|\Sigma|^{n/2}\Gamma_n(n/2)\Gamma_n(m/2)}|L_1|^{(m-n-1)/2}
\prod_{ i<j}^{n}(\ell_i-\ell_j) 
\sum_{k=0}^{\infty}\frac{h^{k}(0)}{k!}\sum_{\kappa\in P^k_m}\int_{V_{n, m}}\mathcal{C}_\kappa (\Sigma^{-1}H_1 L_1 H_1^\top) (H_1^\top d H_1).
\end{align*}
From \eqref{eqn:int-C_HXH} and \eqref{vol_Steifel}, we obtain the desired result.
\end{proof}

If $h(y) =\exp(-y/2)/(2\pi)^{m n / 2}$ in Theorem~\ref{thm:joint-dis}, 
the corresponding joint density function is the same as that in Shimizu and Hashiguchi~\cite{S2021a}.
In the same manner as Shimizu and Hashiguchi~\cite{S2021a}, 
we also provide the distribution function of the largest eigenvalue of $W$ by using a useful lemma from Sugiyama~\cite{S1967}. 
Let $X_1=\mathrm{diag}(1, x_2, \dots, x_n)$, $X_2=\mathrm{diag}(x_2, \dots, x_n)$, where $x_2>\cdots>x_n>0$. 
Sugiyama~\cite{S1967} provided the following lemma as
%The procedure of Sugiyama~(1967) gave as
\begin{align} \label{eqn:Sugiyama-formula}
&\int_{1>x_2>\cdots>x_n>0}|X_2|^{t-(n+1)/2}\mathcal{C}_\kappa(X_1)\prod_{i=2}^{n}(1-x_i)\prod_{i<j}(x_i-x_j)\prod_{i=2}^n dx_i\nonumber\\
=&(nt+k)
\frac{\Gamma_n(n/2)(t)_\kappa\Gamma_n(t)\Gamma_n((n+1)/2)}
{\pi^{n^2/2}(t+(n+1)/2)_\kappa\Gamma_n(t+(n+1)/2)}\mathcal{C}_\kappa (I_n),
\end{align}
where $ \mathrm{Re}(t) >\frac{1}{2}(n-1)$.
The above equation \eqref{eqn:Sugiyama-formula} is a special case of 
\begin{align} \label{eqn:Sugiyama-formula2}
\nonumber
 T(a, b) :=& \int_{1>x_1>x_2>\cdots>x_n>0}
 \mathcal{C}_\kappa(X)|X|^{a-(n+1)/2}|I-X|^{b-(n+1)/2}\prod_{i<j}(x_i-x_j)\prod_{i=1}^ndx_i
 \\
=&\frac{\Gamma_n(n/2) }{\pi^{n^2/2}}\frac{(a)_\kappa}{(a+b)_\kappa}
\frac{\Gamma_n(a)\Gamma_n(b)}{\Gamma_n(a+b)}\mathcal{C}_\kappa(I_n), 
\end{align}
for $X=\mathrm{diag}(x_1, \dots, x_n)$, where
$ \mathrm{Re}(a) >\frac{1}{2}(n-1)$, 
$ \mathrm{Re}(b) >\frac{1}{2}(n-1)$.
The above equation \eqref{eqn:Sugiyama-formula2} is 
equivalent to the well-known formula as 
$$
\int_0^1\cdots \int_0^1 \mathcal{C}_\kappa(X)|X|^{a-(n+1)/2}|I-X|^{b-(n+1)/2}\prod_{i<j}(x_i-x_j)\prod_{i=1}^ndx_i
=n! \; T(a, b)
$$
which is referred to as the Selberg's integral without eigenvalue ordering, see Macdonald~\cite{M2013}. 

\begin{theorem} \label{th:dist}
Let $W\sim \mathcal{EW}_m(n, \Sigma, h)$, where $n<m$. 
Then, the distribution function of the largest eigenvalue $\ell_1$ of $W$ is given as: 
\begin{align}
\label{ell1-hetero}
\mathrm{Pr}(\ell_1<x)
=&
\frac{\pi^{mn/2}\Gamma_n((n+1)/2)}{\Gamma_n((m+n+1)/2)}|x\Sigma^{-1}|^{n/2}
\,{}_1P_1^{(m, n)}\left(h^{(k)}(0): \frac{m}{2}; \frac{m+n+1}{2}; \Sigma^{-1}, xI_n\right). 
\end{align}
\end{theorem}
\begin{proof}
Translating $x_i=\ell_i/\ell_1$ for $i=2, \dots, n $ and using \eqref{eqn:Sugiyama-formula} with  \eqref{eq:join}, the density function of $\ell_1$ is given as 
\begin{align*}
f(\ell_1)
&=\frac{\pi^{n(n+m)/2}|\Sigma|^{-n/2}}{\Gamma_n(n/2)\Gamma_n(m/2)}\int_{1>x_2>\cdots >x_n>0}
|X_2|^{(m-n-1)/2}\prod_{i=2}^n(1-x_i)\prod_{2\le i<j}^n(x_i-x_j)
\mathcal{C}_\kappa(X_2)\\
&\times \sum_{k=0}^\infty\frac{h^{(k)}(0)}{k!}\ell_1^{(mn/2+k-1)} \sum_{\kappa \in P^k_n} \frac{\mathcal{C}_\kappa(\Sigma^{-1})}{\mathcal{C}_\kappa(I_m)} \prod_{i=2}^n dx_i\\
%&=&\frac{\pi^{n(n+m)/2}}{|\Sigma|^{n/2}\Gamma_n(n/2)\Gamma_n(m/2)}\ell_1^{(mn/2+k-1)}\\
%&\times&(nm/2+k)\Gamma_n(n/2)1/\pi^{(n^2/2)}
%\frac{\Gamma_n(m/2, \kappa)\Gamma_n((n+1)/2)}{\Gamma_n((m+n+1)/2, \kappa) }\mathcal{C}_\kappa(I_n)
=&\frac{\pi^{mn/2}\Gamma_n((n+1)/2)}{\Gamma_n((m+n+1)/2)}|\Sigma^{-1}|^{n/2} \sum_{k=0}^\infty (mn/2+k)\ell_1^{mn/2+k-1}\\ \nonumber
&\times\frac{h^{(k)}(0)}{k!}\sum_{\kappa \in P^k_n}\frac{(m/2)_\kappa}{((m+n+1)/2)_\kappa}
\frac{\mathcal{C}_\kappa(\Sigma^{-1})\mathcal{C}_\kappa(I_n)}{\mathcal{C}_\kappa(I_m)}.
\end{align*} 
Finally, by integrating $f(\ell_1)$ with respect to $\ell_1$, we obtain the distribution function of $\ell_1$, and thus the proof is complete.
\end{proof}
%%%%%%%%%%%

Corollary~\ref{cor:Im} shows that when $\Sigma=I_m$, \eqref{ell1-hetero} can be represented in terms of the generalized hypergeometric functions of a single matrix argument of order $n$.
\begin{corollary} \label{cor:Im}
If $\Sigma=I_m$ in Theorem \ref{th:dist}, then we have
\begin{align*}
\mathrm{Pr}(\ell_1<x)
&=\frac{\pi^{mn/2}\Gamma_n((n+1)/2)}{\Gamma_n((m+n+1)/2)}x^{mn/2}
\, {}_1P_1\left(h^{(k)}(0): \frac{m}{2}; \frac{m+n+1}{2}, xI_n\right),
\end{align*}
where ${_1P_1}$ is given in \eqref{eqn:def-1P1}.
\end{corollary}
\begin{proof}
The required result is easily obtained from
\begin{align*}
{}_1P_1^{(m, n)}\left(h^{(k)}(0): \frac{m}{2}; \frac{m+n+1}{2};I_m, xI_n\right)
&=\sum_{\kappa \in P^k_n}\frac{(m/2)_\kappa}{((m+n+1)/2)_\kappa}
\frac{\mathcal{C}_\kappa(I_m)\mathcal{C}_\kappa(I_n)}{\mathcal{C}_\kappa(I_m)}\\
&={}_1P_1\left(h^{(k)}(0): \frac{m}{2}; \frac{m+n+1}{2}, xI_n\right).
\end{align*}
\end{proof}
\begin{corollary}\label{coro:general}
Under the same conditions as in Theorem \ref{th:dist}, the distribution function of $\ell_1$ is also represented by 
\begin{align}
\mathrm{Pr}(\ell_1<x)
\nonumber
&=\frac{\pi^{mn/2}\Gamma_n((n+1)/2)}{\Gamma_n((m+n+1)/2)}
|x\Sigma^{-1}|^{n/2}
{}_1P_1\left(h^{(k)}(0): \frac{n}{2}; \frac{m+n+1}{2}; x\Sigma^{-1}\right)\\ 
&=\frac{\pi^{mn/2}\Gamma_n((n+1)/2)}{\Gamma_n((m+n+1)/2)}
|x\Sigma^{-1}|^{n/2}{}_1P_1\left(h^{(k)}(\mathrm{tr} x\Sigma^{-1}): \frac{m+1}{2}; \frac{m+n+1}{2}; -x\Sigma^{-1}\right).
\label{kum-singular;dist-1}
\end{align}
\end{corollary}
\begin{proof}
This proof is the same as that of Corollary~5 in Shimizu and Hashiguchi~\cite{S2021a}. 
The zonal polynomials are expressed for the length of the partition $p>0$, as 
\begin{align*}
\mathcal{C}_\kappa(I_m)=\frac{2^{2k}k!(m/2)_\kappa \prod_{i<j}^p(2\kappa_i-2\kappa_j-i+j)}{\prod_{i=1}^p(2\kappa_i+p-i)!}, 
\end{align*}
which yields 
$(m/2)_\kappa/\mathcal{C}_\kappa(I_m)=(n/2)_\kappa/\mathcal{C}_\kappa(I_n)$.
Furthermore, if the length of a partition $\kappa$ is $m$, then we have $(n/2)_\kappa=0$, where $m>n$. 
Then, the generalized heterogeneous hypergeometric function ${}_1P_1^{(m, n)}$ in \eqref{ell1-hetero} can be represented as
\begin{align}
\label{ell1-1P1}
\nonumber
&{}_1P_1^{(m, n)}\left(h^{(k)}(0): \frac{m}{2}; \frac{m+n+1}{2}; \Sigma^{-1}, xI_n\right)\\ \nonumber
&=\sum_{k=0}^\infty \frac{h^{(k)}(0)}{k!}\sum_{\kappa \in P^k_n}\frac{(m/2)_\kappa}{((m+n+1)/2)_\kappa}
\frac{\mathcal{C}_\kappa(\Sigma^{-1})\mathcal{C}_\kappa(I_n)}{\mathcal{C}_\kappa(I_m)}\\  
&=\sum_{k=0}^\infty \frac{h^{(k)}(0)}{k!}\sum_{\kappa \in P^k_m}\frac{(n/2)_\kappa \mathcal{C}_\kappa(\Sigma^{-1})}{((m+n+1)/2)_\kappa }={}_1P_1\left(h^{(k)}(0): \frac{n}{2};\frac{m+n+1}{2}; x\Sigma^{-1}\right).
\end{align}
D{\'i}az-Garc{\'i}a and Caro-Lopera~\cite{GarciaCaro2008} provided the Kummer relation of $_1P_1$ as 
\begin{align}
{}_1P_1(h^{(k)}(0): a; c; X)={}_1P_1(h^{(k)}(\mathrm{tr} X):c-a; c;-X)\label{eq:Kummer}
\end{align}
By applying \eqref{eq:Kummer} to \eqref{ell1-1P1}, we obtain the desired result.
\end{proof}

In the case that the elliptical Wishart matrix is nonsingular, Shinozaki et al.~\cite{S2018} gave the distribution of the largest eigenvalue $\ell_1$ in the same manner as Theorem~\ref{th:dist} as
\begin{align}
\label{nonsingular;dist-l1}
\mathrm{Pr}(\ell_1<x)=\frac{\pi^{mn/2}\Gamma_m((m+1)/2)}{\Gamma_m((n+m+1)/2)}|x\Sigma|^{-n/2}{}_{1}P_{1}\left(h^{(k)}(0): \frac{n}{2}; \frac{n+m+1}{2}; x\Sigma^{-1} \right), 
\end{align} 
where $n\geq m$.
Corollary~\ref{coro:general} is a generalized expression of \eqref{kum-singular;dist-1} and (\ref{nonsingular;dist-l1}). 
\begin{corollary}  \label{cor;dist-l1}
Let $W\sim\mathcal{EW}_m(n, \Sigma, h)$. Then, the distribution function of the largest eigenvalue $\ell_1$ of $W$ is given as: 
\begin{align}
 \label{dist-l1}
\mathrm{Pr}(\ell_1<x)=\frac{\pi^{mn/2}\Gamma_t((t+1)/2)}{\Gamma_t((n+m+1)/2)}|x\Sigma|^{-n/2}{}_{1}P_{1}\left(h^{(k)}(0): \frac{n}{2}; \frac{n+m+1}{2}; x\Sigma^{-1} \right), 
\end{align}
where $t=\mathrm{min}(n, m)$. 
\end{corollary}
If $h(y) =\exp(-y/2)/(2\pi)^{m n / 2}$, the function (\ref{dist-l1}) for nonsingular and singular cases coincides with the results of Sugiyama~\cite{S1967} and Shimizu and Hashiguchi~\cite{S2021a}, respectively. 
Namely, the function \eqref{dist-l1} in the Gaussian case is educed to  
\begin{align*}
\mathrm{Pr}(\ell_1<x)=\frac{\Gamma_t((t+1)/2)(\frac{x}{2})^{nm/2}}{\Gamma_t((n+m+1)/2)|\Sigma|^{n/2}}{}_{1}F_{1}\left(\frac{n}{2}; \frac{n+m+1}{2};-\frac{1}{2}\Sigma^{-1} \right), 
\end{align*}
where $t=\mathrm{min}(n, m)$. 

\section{Numerical experiments}
\label{numerical}
In this section, we discuss the numerical computations of \eqref{kum-singular;dist-1} under the matrix variable $t$ and Kotz-type models.
If the generating function $h(x)$ and its $k$-th derivative are given as 
\begin{align}
h(y)&=\frac{\Gamma((mn+\rho)/2)}{(\pi \rho)^{mn/2}\Gamma(\rho/2)}(1+y/\rho)^{-(mn+\rho)/2}, 
\text{ and }\label{eq:geneT}\\
h^{(k)}(y)&=\frac{\Gamma((mn+\rho)/2)(-1)^k((mn+\rho)/2)_k}{(\pi\rho)^{mn/2}\Gamma(\rho/2)\rho^k}(1+y/\rho)^{-((mn+\rho)/2+k)}, \label{eq:generateT}
\end{align}
respectively, then an $m\times n$ random matrix $X$ is said to have a matrix-variate $t$ distribution, 
denoted by $T_{m\times n}(\rho, \Sigma)$.  
The corresponding density function in \eqref{eq:ellip} is also given by
\begin{align*}
g_{\tiny X}(X)=\frac{\Gamma((mn+\rho)/2)}{(\pi\rho)^{mn/2}\Gamma(\rho/2)|\Sigma|^{n/2}}
(1+\mathrm{tr}(X^\top\Sigma^{-1}X)/\rho)^{-(mn+\rho)/2}. 
\end{align*}
We can determine a constant $M$ in Theorem~\ref{thm:conv-rPs} for \eqref{eq:generateT}.
\begin{corollary} \label{cor:h-T}
If the generating function $h(x)$ is in the form of \eqref{eq:geneT} and $m n \le \rho$, 
then the superiority of $| h^{k}(0)|$ for $k=0, 1, \dots$ is evaluated by $M = \pi ^{- m n/2}$ in Theorem~\ref{thm:conv-rPs}.
\end{corollary}
\begin{proof}
Let $a = (mn + \rho)/2$ and the sequence $\{ h^{k}(0) \mid k=0, 1, \dots\}$ monotonically increase for $k$. 
From Staring's formula, it is clear that
\begin{align*}
 \lim_{k \to \infty} \dfrac{\Gamma(a + k)}{\Gamma(a) a^k} = 1.
\end{align*}
Therefore, if $m n \le \rho$, then we have $a \le \rho$ and 
\begin{align*}
| h^{(k)}(0) | &=\frac{\Gamma((mn+\rho)/2)((mn+\rho)/2)_k}{(\pi\rho)^{mn/2}\Gamma(\rho/2)\rho^k}
= \dfrac{\Gamma(a + k)}{\Gamma(a) a^k} \dfrac{\Gamma(a) a^k}{\Gamma(\rho/2) \rho^k} \frac{1}{(\pi \rho)^{mn/2} }\\
&<\dfrac{\Gamma(a + k)}{\Gamma(a) a^k} 
\dfrac{\Gamma(\rho/2)(\rho/2)(\rho/2+1)\cdots (\rho/2+[mn/2])}{\Gamma(\rho/2) \rho^{[mn/2]}}
 \frac{1}{\pi ^{mn/2} }
\\
&\to  \frac{1}{\pi^{mn/2} } \text{ as } k \to \infty,
\end{align*}
where $[x]$ is the Gauss symbol of $x$.
Hence, we can take $M = \pi^{ - mn/2}$ in Theorem~\ref{thm:conv-rPs}.
\end{proof}
Let $X\sim T_{m\times n}(\rho, \Sigma)$, where $m>n$ and $\rho\ge mn$. 
From \eqref{kum-singular;dist-1}, the truncated distribution up to the $K$ th degree of $\ell_1$ of $W=XX^\top$  is given by
\begin{align*}
F_K(x)&=\frac{\Gamma_n((n+1)/2)\Gamma((mn+\rho)/2)}{\Gamma_n((m+n+1)/2)\Gamma(\rho/2)}
|x\Sigma^{-1}|^{n/2} \sum_{k=0}^K\frac{((mn+\rho)/2)_k\rho^{\rho/2}}{(\rho+x \mathrm{tr}\Sigma^{-1})^{(mn+\rho)/2+k}}
\sum_{\kappa\in P^k_m}\frac{((m+1)/2)_\kappa}{((m+n+1)/2)_\kappa}\frac{\mathcal{C}_\kappa(x\Sigma^{-1})}{k!}. 
\end{align*}
We use the algorithm of Hashiguchi et al.~\cite{H2000} for the calculation of zonal polynomials in the above function.
The empirical distribution based on $10^6$ Monte Carlo simulations is represented by $F_{\mathrm{sim}}$. 
The generation of $X\sim T_{m\times n}(\rho, \Sigma)$ is performed according to Theorem~3 of Shinozaki et al.~\cite{S2018}.
Table \ref{table:t_m} indicates several percentile points of correlated and uncorrelated cases. 
We see that $F_K$ has at least two-decimal-place precision. 
%Furthermore, If $\rho\to\infty$, the percentiles of coincide with 

\begin{table}[H]
\caption{Percentile points of truncated distribution~($m=3, n=2, \rho=7$)} \label{table:t_m}\begin{center}
\begin{tabular}{c}
    \begin{minipage}[c]{0.4\hsize}
      \begin{center}
%\tbl{Comparison of acoustic for frequencies for piston-cylinder problem.}
       \captionsetup{labelformat=empty,labelsep=none}
         \subcaption{$\Sigma=\mathrm{diag}(1, 1, 1)$}
{\begin{tabular}{@{}cccc@{}} \toprule
$\alpha$&${{F^{-1}_\mathrm{sim}}}(\alpha)$ &$F^{-1}_{100}(\alpha)$  \\
0.05	 &~1.15& 	1.15\\
0.10	 &~1.61& 	1.61\\
0.50	 &~4.87& 	4.87\\
0.90	 &~14.2& 	14.2\\ 
0.95	 &~19.5& 	19.5\\
\noalign{\smallskip}\hline
\end{tabular}}
 \end{center}
  \end{minipage}
  \begin{minipage}[c]{0.4\hsize}
          \begin{center}
        \captionsetup{labelformat=empty,labelsep=none}
           \subcaption{$\Sigma=\mathrm{diag}(3, 2, 1)$}
{\begin{tabular}{@{}cccc@{}} \toprule
$\alpha$&${F_\mathrm{sim}^{-1}}(\alpha)$ &$F_{140}^{-1}(\alpha)$ \\
0.05	 &~2.15& 	2.16\\
0.10	 &~3.05& 	3.06\\
0.50	 &~9.65& 	9.65\\
0.90	 &~29.5& 	29.5\\ 
0.95	 &~41.0& 	41.1\\ 
\noalign{\smallskip}\hline
\end{tabular}}
        \end{center}
    \end{minipage}
  \end {tabular}
  \end{center}
\end{table}

Next, we  illustrate the computation of \eqref{kum-singular;dist-1} with the Kotz-type model.
Caro-Lopera~\cite{C2008} classified the Kotz-type distribution into three subfamilies: Kotz types I, II, and III. 
The generator function for the Kotz type I distribution is given by
\begin{align}
\label{eq:hkotzdist}
h(x)=\frac{\theta^{\frac{2q +mn-2}{2}}\Gamma\left(\frac{mn}{2} \right)}{\pi^{mn/2}
\Gamma\left(\frac{2q+mn-2}{2}\right)}x^{q-1}\exp(-\theta x), 
\end{align}
where $\theta>0$ and $2q+mn>2$. 
\begin{corollary} \label{cor:h-Kotz}
If the generating function $h(x)$ is in the form of \eqref{eq:hkotzdist} with $|\theta|< 1$, 
then the superiority of $| h^{k}(0)|$ for $k=0, 1, \dots$ is evaluated by $M = \pi^{- m n/2}$ in Theorem~\ref{thm:conv-rPs}.
\end{corollary}
\begin{proof}
Applying the Leibniz rule to \eqref{eq:hkotzdist}, we can evaluate the  superiority of $| h^{k}(0)|$ for $k=0, 1, \dots$ as
\begin{align*}
|h^{(k)}(0)| &=\frac{\theta^{2q+mn-2/2}\Gamma(m n/2)}{\pi^{m n/2}\Gamma((2q+mn-2)/2)}{_kC_{q-1}}\theta^{k-q+1} \le \frac{\theta^{mn/2+k}k^{q-1}}{\pi^{m n/2}}
\\
&\to  \frac{1}{\pi ^{mn/2} } \text{ as } k \to \infty.
\end{align*}
because $\Gamma(mn/2)/\Gamma((2q+mn-2)/2) < 1$ and  ${_kC_{q-1}}\le k^{q-1}$ for $k=0,1, \dots$. 
Hence, we can take $M =\pi^{- m n/2}$ in Theorem~\ref{thm:conv-rPs}.
\end{proof}
The $k$-th derivatives for the Kotz type II and III distributions can also be obtained by Fa\`a di Bruno's formula found in Caro-Lopera~\cite{C2008, C2010}.
If we set $\theta=1/2$ and $q=2$ in \eqref{eq:hkotzdist}, the distribution \eqref{kum-singular;dist-1} is reduced to
\begin{align}
\mathrm{Pr}(\ell_1<x)
&=\frac{\Gamma_n((n+1)/2)\Gamma(mn/2)}{\Gamma_n((n+m+1)/2)\Gamma(mn/2+1)}
\left(\frac{1}{2}\right)^{mn/2+1}|x\Sigma^{-1}|^{n/2}\mathrm{etr}(-x\Sigma^{-1}/2)\nonumber\\
&\times \sum_{k=0}^\infty\left(\frac{x}{2}\right)^k(\mathrm{tr} x\Sigma^{-1}-2k)
\sum_{\kappa\in P_m^k}\frac{((m+1)/2)_\kappa}{((m+n+1)/2)_\kappa}\frac{\mathrm{C}_\kappa(\Sigma^{-1})}{k!}. \label{eq:distKsig}
\end{align}
The generation of random numbers for the matrix-variate Kotz type I distribution with parameters $\theta=1/2$ and $q=2$ is based on Definition~5 and Theorem~7 of Kollo and Roos~\cite{K2005}.
Fig \ref{fig:dist-k} shows the comparison of the truncated distribution up to the $70$ th degree for \eqref{eq:distKsig} and $F_{\mathrm{sim}}$ for the parameters $n=2$ and $\Sigma=\mathrm{diag}(3, 2, 1)$. 
We observe that the truncated distribution is very close to the empirical distribution $F_\mathrm{sim}$.
The $95$ percentage points of both their distribution are $29.0$.
 \begin{figure}[H]
\begin{center}
\includegraphics[width=8cm]{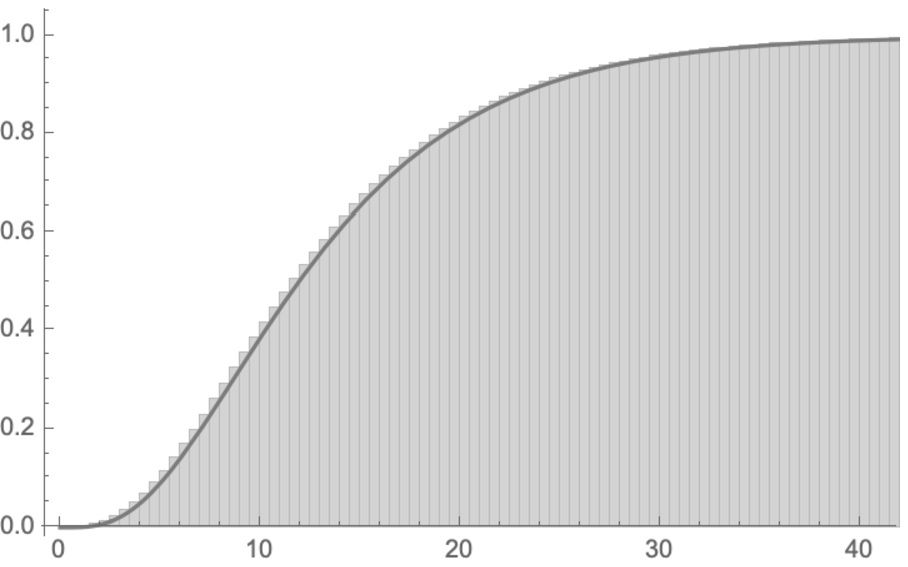}
\rlap{\raisebox{5.2cm}{\kern-8.1cm{\small $F_{70}(x)$}}}
\rlap{\raisebox{.3cm}{\kern0.0cm{\small $x$}}}
\caption{$\Sigma=\mathrm{diag}(3,2,1)$, $n=2$} \label{fig:dist-k}
\end{center}
\end{figure}





\end{document}